\documentclass{birkjour}
\usepackage[all]{xy}
\usepackage[utf8]{inputenc}
\usepackage{hyperref}
\usepackage{amsmath,amssymb}
\usepackage{enumerate,multicol}
\newtheorem{thm}{Theorem}[section]
\newtheorem{cor}[thm]{Corollary}

\newtheorem{prop}[thm]{Proposition}
\theoremstyle{definition}

\theoremstyle{remark}

\numberwithin{equation}{section}

\DeclareMathOperator{\Ricci}{Ric}
\DeclareMathOperator{\Hess}{\mathrm{Hess}}

\def\rm#1{\mathrm{#1}}

\def\bb#1{\mathbb{#1}}

\usepackage[all]{xy}

\def\action#1#2{\,_{#1}{\times}_{#2}\,}

\begin{document}
	
	%-------------------------------------------------------------------------
	% editorial commands: to be inserted by the editorial office
	%
	%\firstpage{1} \volume{228} \Copyrightyear{2004} \DOI{003-0001}
	%\seriesextraline{This is the Concrete Title of this Book\br H.E. R and S.T.C. W, Eds.}
	%\firstpage{1}

	\title[Positive Ricci on fiber bundles]{Positive Ricci curvature on fiber bundles with compact structure group}

	%----------Author 1
	\author[Cavenaghi]{Leonardo Francisco Cavenaghi}
	
	\address{%
		Instituto de Matem\'atica e Estat\'istica -- USP, R. do Mat\~ao, Vila Universit\'aria, 1010, 05508-090, S\~ao Paulo - SP, Brazil}
	
	\email{leonardofcavenaghi@gmail.com}
	
	\thanks{This work was financially supported by CNPq 131875/2016-7, 404266/2016-9 and by FAPESP 2017/24680-1.}
	%----------Author 2
	\author[Speran\c ca]{Llohann Dallagnol Speran\c ca}
	\address{Instituto de Ciência e Tecnologia -- Unifesp, Avenida Cesare Mansueto Giulio Lattes, 1201,  12247-014, São José dos Campos, SP, Brazil}
	\email{speranca@unifesp.br}
	%\subjclass{primary 99Z99; Secondary 00A00}
	
	
	\date{}
	
	\begin{abstract}
		The aim of this paper is to present a direct and simple proof of a result concerning the existence of metrics of positive Ricci curvature on the total space of fiber bundles with compact structure groups. In particular, it  generalizes and puts in a unified framework the results in Nash \cite{nash1979positive} and Poor \cite{poor1975some}. 
		With the intention of disseminating this result, we apply it to build new examples of manifolds with positive Ricci curvature, including bundles whose base consists of gradient shrinking Ricci solitons. 
	\end{abstract}
	
	%%% ----------------------------------------------------------------------
	\maketitle
	%%% ----------------------------------------------------------------------
	%\tableofcontents
	\section{Introduction}
	
	The amount of examples of manifolds with positive sectional curvature is small compared to those of manifolds with non--negative sectional curvature (see \cite{ziller2007examples,Zillerpositive}). However, there is no theorem  distinguishing  these two classes in the compact simply connected case. In this scenario, the understanding of special classes of examples has an special importance, since it can lead to new insights regarding possible obstructions to move from one class to the other. An intermediate step consists in constructing examples of manifolds with positive Ricci curvature.
	
	%Since Milnor discovered the existence exotic spheres, questions on which of these spheres admit metrics of positive curvatures were exploited. Until these days, it is not known an example of an exotic sphere with positive sectional curvature.
	
	%On the other hand, 
	
	In this direction, Nash \cite{nash1979positive} and  Poor \cite{poor1975some} proved that the total spaces of some classes of fiber bundles admit metrics of positive Ricci curvature: for instance, of linear sphere bundles, vector bundles and principal bundles over manifolds with positive Ricci curvature. In particular, their results provide metrics of positive Ricci curvature on exotic spheres. Other interesting results along the same lines are in \cite{WB,grove, tusch, Sch}.
	
	Here we present an elementary proof for the following natural generalization of Nash and Poor results:
	
	\begin{thm}\label{horizontal}
		Let $F\hookrightarrow M\stackrel{\pi}{\to}B$ be fiber bundle with compact total space and structure group $G$. Suppose that
		\begin{enumerate}
			\item $B$ has a metric of positive Ricci curvature,
			\item A principal orbit of the $G$-action on $F$ has finite fundamental group,
			\item $F$ has a $G$-invariant metric such that the submersion metric on $F^{\mathrm{reg}}/G$ has $\Ricci_{F^{\mathrm{reg}}/G}\ge 1$. 
		\end{enumerate}
		Then $M$ admits a metric of positive Ricci curvature.
	\end{thm}
%	\begin{rem}
		We point out that, in the case of a principal $G$-bundle, the metric resulting from Theorem \ref{horizontal} can be made $G$-invariant\footnote{The authors would like to thank W. Tuschmann for questioning this.}. This can be done by applying the \textit{canonical deformation} to a Kaluza--Klein metric (we refer to the proof of Theorem \ref{conjecture} for details).
%	\end{rem}
	
	We use Theorem \ref{horizontal} to build new examples of manifolds with positive Ricci curvature, which are the main cornerstone of the paper. The method we use naturally extends the range of examples in \cite{searle2015lift} and \cite{SperancaCavenaghiPublished}, since one can use the examples therein as both bases and fibers. The new examples also include bundles whose bases are gradient shrinking Ricci solitons. The detailed account is given in Section \ref{section:examples}. Furthermore, Theorem \ref{horizontal} puts in a common ground the Theorems of Nash and Poor:
	%, since they can be obtained as a corollary of Theorem \ref{horizontal}:
	
	\begin{cor}[\cite{nash1979positive}, Theorem 3.5]
		Let $G/H \hookrightarrow E \to M$ be a fiber bundle with compact structure group $G$ such that $M$ is compact and admits a metric $g$ with $\Ricci_g > 0$. If $G\big/ H$ admits a metric with $\Ricci \ge 1$ (equivalently, if $\pi_1
		(G/ H)$ is finite), then $E$ admits a metric of positive Ricci curvature.
	\end{cor}
	
	\begin{cor}[\cite{nash1979positive}, Corollary 3.6]
		Let $S^n \hookrightarrow E \to M$, $n \ge 2,$ be a sphere bundle whose base is compact and has a metric of positive Ricci curvature. Then $E$ admits a metric of positive Ricci curvature.
	\end{cor}
	
	\begin{cor}[\cite{nash1979positive}, Theorem 3.8]
		Let $F \hookrightarrow E \to M$ be a bundle with compact strcture group $G$ and $M$ be a compact manifold admitting a metric $g$ with $\Ricci_g > 0$. If $G$ possesses finite fundamental group and acts freely on $F$, which is compact and such that $F/G$ admits a metric with $\Ricci \ge 1$, then $E$ admits a metric of positive Ricci curvature.
	\end{cor}
	
	\begin{cor}[\cite{poor1975some}, Main Theorem]
		Let $M$ and $F$ be compact Riemannian manifolds such that the Ricci curvature of $M$ and the sectional curvature of $F$ are positive. Suppose $\pi : FM \to M$ is a bundle with fiber $F$ for which the structure group $G$ acts on F by isometries. Then $FM$ admits a metric of positive Ricci curvature. %in which the fibers of $\pi$ are totally geodesic and mutually isometric in $FM$. 
	\end{cor}
	
	Theorem \ref{horizontal} follows from the results below, which can be found in the literature:
	
	\begin{thm}\label{conjecture}
		Let $F\hookrightarrow M\stackrel{\pi}{\to}B$ be a bundle with compact connected total space and structure group $G$. Suppose that
		\begin{enumerate}
			\item $B$ has a metric of positive Ricci curvature,
			\item $F$ has a $G$-invariant metric of positive Ricci curvature.
		\end{enumerate}
		Then $M$ admits a metric of positive Ricci curvature.
	\end{thm}

	\begin{thm}[Theorem A, \cite{searle2015lift}]\label{thm:searle}
		Let $(M,g)$ be a compact connected Riemannian manifold with an isometric  effective action by a compact connected Lie group $G$. If
		\begin{enumerate}
			\item A principal orbit has finite fundamental group,
			\item $\Ricci_{M/G}\geq 1$ for the orbital distance metric,
		\end{enumerate}
		then $M$ admits a $G$--invariant metric of positive Ricci curvature.
	\end{thm}
	
	Theorem \ref{conjecture} is proved using the classical {canonical variation} for fiber bundles and a `{quadratic trick}'. Although this result appears in the literature (see \cite[Theorem 2.7.3, p. 100]{gw}) our correspondences with several specialists revealed that it is  not as popular as it should be, which was one of our motivations for writing this paper. 
	
	Another result, obtained as a direct consequence of the proof of Theorem \ref{conjecture} is the following:
	\begin{cor}\label{thm:trivial}
		Let $(M,g), (F,g_F)$ be compact Riemannian manifolds with positive Ricci curvature. Then, for any smooth function $f: M\to \mathbb{R},$ there is $a\in \bb R$ big enough such that the warped product $M\times_{e^{2f-a}}F$ has positive Ricci curvature.
	\end{cor}
	
	A complete Riemannian metric $g$ on a smooth manifold $M$ is called \textit{Ricci soliton} if there exists a smooth vector field $X$ on $M$ and a constant $\rho \in \mathbb{R}$ such that the following equation holds
	\begin{equation}\label{eq:defi} \Ricci(g) + \dfrac{1}{2}\mathcal{L}_Xg = \rho g,\end{equation}
	where $\mathcal{L}$ denotes the Lie derivative.	
	If $X = \nabla f$ where $f$ is some smooth function, the equation \eqref{eq:defi} becomes
	\begin{equation}\label{eq:gsrs}
	\Ricci(g) + \Hess f = \rho g,
	\end{equation}
	and the Ricci soltion is called \textit{gradient Ricci soliton}. 
	
	Ricci solitons are named in the literature according to the sign of $\rho.$ A Ricci soliton is called \textit{steady} if $\rho = 0,$ \textit{expanding} if $\rho < 0$, and \textit{shrinking} if $\rho > 0.$
	
	Ricci solitons were introduced by Hamilton in \cite{hamiltons} as the particular solutions for the Ricci flow $\partial g /\partial t = -2\Ricci(g)$ whose metric evolves only by diffeomorphisms. Their importance was further demonstrated in the works of Perelman \cite{perelman2002entropy,perelmeme} where his classification of $3$--dimensional gradient shrinking Ricci solitons led to the solution of the Poincaré conjecture, and in \cite{sesum}, where it is proved that gradient Ricci solitons are precisely the blow--up limits of the Ricci flow.
	
	Compact Shrinking Ricci Solitons (SRS) with positive Ricci curvature were extensively studied in the literature (see \cite{cao2009recent,dancer}). In particular, using warped product metrics and Corollary \ref{thm:trivial} one can use such Ricci solitons to build new examples of manifolds of positive Ricci curvature (see Section \ref{section:examples}).
	
	\section{Proof of the main theorems}
	
	Let $\pi : (F,g_F)\hookrightarrow (M,g) \to (B,g_B)$ be a Riemannian submersion. The \textit{canonical variation} $\widetilde g$ of the metric $g$ is defined as:
	\begin{equation}
	\widetilde g(X+V,Y+W) := e^{2t}g(V,W) + g(X,Y), ~t \in (-\infty,\infty),
	\end{equation}
	$\forall X,Y \in \mathcal{H},\forall  V,W \in \mathcal{V},$ where $\mathcal{H}, \mathcal{V}$ denote the horizontal and vertical space of the submersion $\pi$. 
	
	\begin{prop}\label{prop:curvatures}
		Let $\pi : F \hookrightarrow (M,g) \to B$ be a Riemannian submersion with totally geodesic fibers. Let $\widetilde K,~K,~K_B,~K_F$ denote the non--reduced sectional curvatures of $\widetilde g,~g_B,~g_F$, the canonical variation metric of $g,$ the submersion metric on $B$, and the metric on $F$, respectively. Then, if $X,Y,Z \in \mathcal{H},$ and  $V,W \in \mathcal{V},$
		\begin{enumerate}
			\item $\widetilde K(X,Y) = K_B(\pi_{\ast}X,\pi_{\ast}Y)(1-e^{2t}) + e^{2t}K(X,Y),$\\
			\item $\widetilde K(X,V) = e^{4t}\|A^*_XV\|^2,$\\
			\item $\widetilde K(V,W) = e^{2t}K(V,W),$\\
			\item $\widetilde R(X,Y,Z,W) = e^{2t}g((\nabla_XA)_YZ,W).$
		\end{enumerate}
	\end{prop}
	\begin{proof}
		See \cite[Section 2.1.3, p. 52]{gw}.
	\end{proof}
	
	We now prove Theorem \ref{conjecture}.
	\begin{proof}[Proof of Theorem \ref{conjecture}]
		The main idea is to control the curvature using the discriminant of a quadratic polynomial (this is sometimes the `quadratick trick'), while applying a {canonical variation}. Firstly, consider on $M$ a metric $g$ such that the fibers $F$ are totally geodesic (see \cite[proposition 2.7.1, p. 97]{gw}) and let $\widetilde g$ be a canonical variation of the metric $g$. 
		
		Take $X \in \mathcal{H}, V \in \mathcal{V}, \|X\|_{\widetilde g} = \|V\|_{\widetilde g} = 1$ and let $\lambda\in \mathbb{R}.$ Define the polynomial $p(\lambda) := \widetilde \Ricci(X+\lambda V).$ Take a $g$--orthonormal basis  $\{\widetilde e_i := e_i\}_{i=1}^k$ for $\mathcal{H}$ and a $\widetilde g$--orthonormal basis $\{\widetilde e_j := e^{-t}e_j\}_{j=k+1}^n$ for $\mathcal{V}.$ Then,
		\[\widetilde R(X+\lambda V,\widetilde e_r,\widetilde e_r,X+\lambda V) = \widetilde K(X,\widetilde e_r) + \lambda^2\widetilde K(V,\widetilde e_r) + 2\lambda \widetilde R(X,\widetilde e_r,\widetilde e_r,V),\]
		$\forall r\in \{1,\ldots,n\}.$
		
		Summing in $i,j,$ and using the formulae on Proposition \ref{prop:curvatures}, by recalling that the fibers are totally geodesic, one has:
		\begin{multline}
		p(\lambda) = (1-e^{2t})\Ricci_B(X) + e^{2t}\Ricci^{\mathcal{H}}(X)+e^{2t}\sum_{j=k+1}^{n}\|A_X^*e_j\|^2 \\+\lambda^2\left\{e^{2t}\sum_{i=1}^k\|A_{e_i}V^*\|^2 + \Ricci_F(V)\right\}+ 2\lambda e^{2t}\sum_{i=1}^kg\left((\nabla_{e_i}A)_Xe_i,V\right).
		\end{multline}
		The discriminant $\Delta_t(X,V)$ of the last expression, when seen as a second order polynomial, is given by:
		\begin{multline*}\dfrac{1}{4}\Delta_t(X,V) = e^{4t}\left(\sum_{i=1}^kg\left((\nabla_{e_i}A)Xe_i,V\right)\right)^2-\\
		\left((1-e^{2t})\Ricci_B(X) + e^{2t}\Ricci^{\mathcal{H}}(X)+e^{2t}\sum_{j=k+1}^{n}\|A_X^*e_j\|^2\right)\\
		\left(e^{2t}\sum_{i=1}^k\|A_{e_i}V^*\|^2 + \Ricci_F(V)\right)
		\end{multline*}

		We claim that there exist $t < 0$ such that, for every $X\in \mathcal{H},V \in \mathcal{V}, \|X\| = \|V\| = 1,$ one has
		$\Delta_t(X,V) < 0.$ In fact, suppose on the contrary that, for each $n \in \mathbb{N}$, there are $X_n,V_n$ as in the hypotheses satisfying $\Delta_{-n}(X_n,V_n) \geq 0.$ By passing to a subsequence, if necessary, when $n \to \infty$ one gets:
		\[-\Ricci_{F}(V)\Ricci_{B}(X) \geq 0,\]
		where $V_n \to V, X_n \to X$. The proof is concluded since $\Ricci_{F}(V)\Ricci_{B}(X) > 0$ by hypotheses.
	\end{proof}
	
	\begin{proof}[Proof of Theorem \ref{horizontal}]
		By the assumption that $F/G$ admits a metric of positive Ricci curvature and the orbits of $F$ on $G$ have finite fundamental group, by Theorem \ref{thm:searle}, $F$ admits a $G$-invariant metric of positive Ricci curvature. Theorem \ref{conjecture} now shows that $M$ admits a metric of positive Ricci curvature.
	\end{proof}
	
	\section{Examples}
	\label{section:examples}
	
	Let $B,F$ be manifolds satisfying the hypotheses in Theorem \ref{thm:searle}. Given a principal $G$--bundle
	\[\pi : G \hookrightarrow P \to B,\]
	Theorem \ref{horizontal} guarantees the existence of a metric with positive Ricci curvature on the associated bundle $\mathrm{pr} : F\hookrightarrow P\times_G F\to B$ (where $P\times_G F$ stands for the quotient of $P\times F$ by the $G$-action $g(p,f)=(pg^{-1},gf)$). Here we combine this idea with the constructions in \cite{SperancaCavenaghiPublished} to produce new examples of bundles with positive Ricci curvature. With this aim, we recall the construction in \cite{SperancaCavenaghiPublished}.
	
	Consider a compact connected principal bundle $G\hookrightarrow P \to M$ with a principal action $\bullet.$ Assume that there is another action on $P$, which we denote by $\star$,  that commutes with $\bullet$. This makes $P$ a $G\times G$--manifold. If one assumes that $\star$ is free, one gets a \textit{$\star$-diagram} of bundles:
	\begin{equation}\label{eq:CD}
	\begin{xy}\xymatrix{& G\ar@{..}[d]^{\bullet} & \\ G\ar@{..}[r]^{\star} & P\ar[d]^{\pi}\ar[r]^{\pi'} &M'\\ &M&}\end{xy}
	\end{equation}
	In \ref{eq:CD}, $M$ is the quotient of $P$ by the $\bullet$-action and $M'$ is the quotient of $P$ by the $\star$-action.
	
	Since $\bullet$ and $\star$ commute, $\bullet$ descends to an action on $M'$ and $\star$ descends to an action on $M$. We denote the orbit spaces of these actions by $M'/\bullet$ and $M/\star$, respectively. Corollary 5.2 in \cite{SperancaCavenaghiPublished} implies that one can choose a Riemannian metric $g'$ on $M'$ such that the orbit spaces $M/\star$ and $M'/\bullet$ are isometric, as metric spaces. Furthermore, it can be shown that the orbits of the $\star$-action on $M$ have finite fundamental group if, and only if, the orbits of the $\bullet$-action does \cite[Theorem 6.4]{SperancaCavenaghiPublished}. Therefore, $M$ has a $\star$-invariant metric satisfying the hypotheses in Theorem \ref{thm:searle} if, and only if, $M'$ has a $\bullet$--invariant metric satisfying the same hypotheses. We gather this information in a Proposition:
	\begin{prop}\label{prop:ajuda}
		Let $P,M, M'$ be as in diagram \eqref{eq:CD}. Then there exists a $\star$-invariant metric on $M$ satisfying the hypotheses of Theorem \ref{thm:searle} if, and only if, $M$ has a $\bullet$-invariant metric satisfying the same hypotheses. 
	\end{prop}
	
	The example that inspired the constructions in \cite{SperancaCavenaghiPublished} was the Gromoll--Meyer $\star$-diagram introduced by C. Durán in \cite{duran2001pointed}.
%	\begin{example}
		Recall that the Lie group $Sp(2)$ is defined by
		\begin{equation}\label{eq:Sp2}
		Sp(2) = \left\{\begin{pmatrix} a & c \\b & d\end{pmatrix}\in S^7\times S^7~ \Big| ~a\bar{b} + c\bar{d} = 0\right\},
		\end{equation} 
		where $a,b,c,d\in \bb H$, the set of quaternions with usual conjugation and multiplication, and $S^7\subset \bb H^2$ is the unitary sphere with respect to the norm $\|(a,b)^T\|^2=\bar aa+\bar bb$.
		The projection $\pi:Sp(2)\to S^7$ of an element of $Sp(2)$ onto its first row is a principal $S^3$--bundle with principal action:
		\begin{equation}\label{eq:GMprincipalaction}
		\begin{pmatrix} 
		a & c \\
		b & d 
		\end{pmatrix}\bar q = \begin{pmatrix}
		a & c\\
		b\overline{q} & d\overline{q}
		\end{pmatrix}.
		\end{equation}
		Gromoll--Meyer \cite{gromoll1974exotic} introduced the $\star$--action
		\begin{equation}\label{eq:GMstaraction}
		q \begin{pmatrix} 
		a & c \\
		b & d 
		\end{pmatrix} = \begin{pmatrix} 
		qa\overline{q} & qc\overline{q} \\
		qb & qd 
		\end{pmatrix},
		\end{equation}
		whose quotient is an exotic sphere of dimension $7$, concluding their celebrated result on the existence of an exotic sphere with non-negative sectional curvature.
		The corresponding action on $S^7$ can be read from the first row of \eqref{eq:GMprincipalaction}:
		\begin{equation}\label{eq:GMactiondown}
		q\cdot\begin{pmatrix}x\\ y	\end{pmatrix}=\begin{pmatrix}qx\bar q\\ qy\bar q\end{pmatrix}
		\end{equation}
		The actions given by \eqref{eq:GMprincipalaction}, \eqref{eq:GMstaraction} produce the following star--diagram
		\begin{equation}\label{cd:duran}
		\begin{xy}\xymatrix{& S^3\ar@{.}[d]^{\bullet} & \\ S^3\ar@{..}[r]^{\star} &Sp(2)\ar[d]^{\pi}\ar[r]^{\pi'} &\Sigma^7\\ &S^7&}\end{xy}
		\end{equation}
		which was used in \cite{duran2001pointed} to geometrically realize an explicit clutching diffeomorphism $\hat b : S^6\to S^6$ for $\Sigma^7=Sp(2)/\star$.
%	\end{example}
	
	%such that $\Ricci_{M/G}\geq 1$ then there exists a $\bullet$--invariant metric on $M'$ with the same property.
	
	The star--diagram construction was used in \cite[Theorem 1.1]{SperancaCavenaghiPublished} to produce several manifolds satisfying the hypotheses in Theorem \ref{thm:searle}, which we recall here: 
	
	\begin{thm}\label{thm:llohann}
		Let $\Sigma^7$ and $\Sigma^8$ be any homotopy spheres of dimensions 7 and 8, respectively; $\Sigma^{10}$ be any homotopy 10-sphere which bounds a spin manifold;  $\Sigma^{4m{+}1},\Sigma^{8m+5}$ be Kervaire spheres of dimensions $4m{+}1,8m{+}5$, respectively. Then, the following manifolds admit an explicit realization as in diagram \eqref{eq:CD} and satisfy the hypotheses of Theorem \ref{thm:searle}:
		\begin{enumerate}[$(i)$]
			\item $M^7\#\Sigma^7$, where $M^7$ is any 3-sphere bundle over $S^4$ 
			\item $M^8\#\Sigma^8$, where $M^8$ is either a 3-sphere bundle over $S^5$ or a 4-sphere bundle over $S^4$
			\item $M^{10}\#\Sigma^{10}$, where:
			\begin{enumerate}[$(a)$]
				\item $M^{10}=M^8\times S^2$ with  $M^8$  as in item $(ii)$
				\item  $M^{10}$ is any 3-sphere bundle over $S^7$, 5-sphere bundle over $S^5$ or 6-sphere bundle over $S^4$				
			\end{enumerate}
			\item $M^{4m+1}\#\Sigma^{4m+1}$ where \label{item:5} 
			\begin{enumerate}[$(a)$]
				\item $S^{2m}\hookrightarrow M^{4m+1}\to S^{2m+1}$ is the sphere bundle associated to any  multiple\footnote{\label{footnote1}That is, a bundle whose transition function $\alpha:S^{n-1}\to G$ is a multiple of $\tau_{2m}:S^{2m}\to O(2m+1)$, $\tau_m^\bb C:S^{2m}\to U(m)$ or $\tau_m^\bb H:S^{4m+2}\to Sp(m)$, for $G=O(2m),U(m+1)$ or $Sp(m)$, the transition functions of the orthonormal frame bundle and its reductions, respectively.} of $O(2m{+}1)\hookrightarrow O(2m{+}2)\to S^{2m+1}$, the frame bundle of $S^{2m+1}$
				\item $\bb C \rm P^{m}\hookrightarrow M^{4m+1}\to S^{2m+1}$ is the $\bb C\rm P^m$-bundle associated to any  multiple of the bundle of unitary frames $U(m)\hookrightarrow U(m+1)\to S^{2m+1}$
				\item $M^{4m+1}=\frac{U(m+2)}{SU(2)\times U(m)}$	
			\end{enumerate}
			\item $(M^{8r+k}\times N^{5-k})\#\Sigma^{8r+5}$ where $N^{5-k}$ is any manifold with positive Ricci curvature and
			\begin{enumerate}[$(a)$]
				\item $S^{4r+k-1}\hookrightarrow M^{8r+k}\to S^{4r+1}$ is the $k$-th  suspension of the unitary tangent $S^{4r-1}\hookrightarrow T_1S^{4r{+}1}\to S^{4r+1}$,
				\item for $k=1$, $\bb H\rm P^{m}\hookrightarrow M^{8m+1}\to S^{4m+1}$ is the $\bb H\rm P^m$-bundle associated to any  multiple of $Sp(m)\hookrightarrow Sp(m+1)\to S^{4m+1}$
				\item for $k=0$, $M=\frac{Sp(m+2)}{Sp(2)\times Sp(m)}$
				\item for $k=1$, $M=M^{8m+1}$ is as in item  $(\ref{item:5})$
			\end{enumerate}
		\end{enumerate}	
	\end{thm}
	
	Now we use the previous discussion and results to produce new examples of manifolds with positive Ricci curvature. In what follows, we denote by $F\hookrightarrow P\times_G F\to B$ the associated bundle to $G\hookrightarrow P\to B$ with fiber $F$.
	
	\subsection{Bundles associated to $\star$-diagrams}
	%    \begin{enumerate}
	%        \item 
	Consider $P,M,M',G$ as in diagram \eqref{eq:CD} and assume that $M$ satisfies the hypotheses in Theorem \ref{thm:searle}. Then, according to Proposition \ref{prop:ajuda},  $M'$ satisfies the same hypotheses. Therefore Theorem \ref{horizontal} guarantees the existence of metrics of positive Ricci curvature in the following bundles:
	\begin{enumerate}
		\item The associated bundle $M \hookrightarrow P\times_GM \to M$ to $\pi : P\to M,$
		\item The associated bundle $M' \hookrightarrow P\times_GM' \to M$ to $\pi : P \to M,$
		\item The associated bundle $M \hookrightarrow P\times_GM' \to M'$ to $\pi' : P\to M',$
		\item  The associated bundle $M' \hookrightarrow P\times_GM' \to M'$ to $\pi' : P\to M'$.
	\end{enumerate}
	
	Here we are considering $M,M'$ with the $G$-actions descending from $\star,\bullet$, respectively. Note that $M'$ can be taken as any of the manifolds given by Theorem \ref{thm:llohann} and $M$ as its ``counterpart'' in diagram \eqref{eq:CD}. 
	%	(i.e, as the other base manifold that completes the diagram \eqref{eq:CD}). 
	
	%	\item 
	\subsubsection{}	Now we explore a little further the construction in the previous paragraph.
	To begin with, consider   Duran's $\star$-diagram diagram \eqref{cd:duran}, i.e., the $\star$-diagram  with $G = S^3, P = Sp(2)$, $M = S^7, M' = \Sigma^7$. Let $B=S^7,\Sigma$ endowed with its respective action and equipped with an invariant  metric with positive Ricci curvature. Consider its respective principal bundle 
	%	hypotheses in Theorem \ref{thm:searle} and consider a principal bundle  
	$\mathrm{pr} : S^3\hookrightarrow Sp(2) \to B.$ 
	\begin{enumerate}[(i)]
		\item Let $Sp(2) \hookrightarrow Sp(2)\times_{S^3}Sp(2) \to B$ be the associated bundle to $\mathrm{pr} : S^3 \hookrightarrow Sp(2) \to B$ . By Theorem \ref{horizontal}, this bundle admits a metric of positive Ricci curvature.
		\item 
		%We can assume, for instance, that $B$ is either $S^7$ or $\Sigma^7$, with $P=Sp(2)$  equipped with the principal $S^3$-action from $\eqref{cd:duran}$. 
		%Thus, , which become $S^3$-manifolds when equipped with the actions induced by $\star$ and $\bullet$, respectively. 
		%
		%  In both cases, the $S^3$--manifold $B$ can be regarded as the fiber of associated bundles to the principal bundles $\pi : Sp(2) \to S^7$ or $\pi':  Sp(2) \to \Sigma^7$ described in diagram \eqref{cd:duran}, producing new examples of fiber bundles with positive Ricci curvature.
		As above, the following associated bundles admit metrics of positive Ricci curvature:
		\begin{enumerate}[(a)]
			\item  $S^7 \hookrightarrow  Sp(2)\times_{S^3}S^7 \to \Sigma^7$;
			\item  $\Sigma^7 \hookrightarrow Sp(2)\times_{S^3}\Sigma^7 \to \Sigma^7$;
			\item  $\Sigma^7\hookrightarrow Sp(2)\times_{S^3}\Sigma^7\to S^7$;
			\item  $S^7 \hookrightarrow Sp(2)\times_{S^3}S^7 \to S^7$.
		\end{enumerate}

		\item
		Based on the descriptions given by the items $(i), (ii),$ one can build, for example, the following ladder of bundles with positive Ricci curvature:
		\begin{equation}\label{cd:stairs}
		\xymatrix@C=1.2em@R=1.4em{
			Sp(2) \ar[d] &\\
			Sp(2)\times_{S^3}Sp(2) \ar[d] \\
			S^7  \ar[r] &   Sp(2)\times_{S^3}S^7 \ar[d]\\
			&  \Sigma^7 \ar[r] &Sp(2)\times_{S^3}\Sigma^7 \ar[d]\\
			& &  \Sigma^7 \ar[r] & Sp(2)\times_{S^3}\Sigma^7 \ar[d]\\
			& & & S^7}
		\end{equation}
		It is noteworthy that the bundle in each step can be incorporated to the one in the next lower step, resulting  in a clustered bundle. For instance \footnote{we use the notation $Sp(2)\action{*_1}{*_2}Sp(2)$ in order to indicate which $S^3$ action is taken into account. Namely, $Sp(2)\action{*_1}{*_2}Sp(2)$ denotes the quotient of $Sp(2)\times Sp(2)$ by the $S^3$-action $q(A,B)=(q*_1A,q*_2B)$}:
		\begin{enumerate}[$(a)$]
			\item $Sp(2)\action{\star}{\star}(Sp(2)\action{\bullet}{\bullet}Sp(2))\to \Sigma^7$ is a bundle whose typical fiber is the total space of the bundle $Sp(2)\action{\bullet}{\bullet}Sp(2)\to S^7$;
			\item $Sp(2)\action{\star}{\bullet}(Sp(2)\action{\star}{\star}(Sp(2)\action{\bullet}{\bullet}Sp(2)))\to \Sigma^7$ is a bundle  whose typical fiber is the bundle in item $(a)$;
			\item $Sp(2)\action{\star}{\bullet}(Sp(2)\action{\star}{\bullet}(Sp(2)\action{\star}{\star}(Sp(2)\action{\bullet}{\bullet}Sp(2))))\Sigma^7$ is a bundle  whose typical fiber is the bundle in item $(b)$;
			\item the bundle represented by the whole ladder is: \[\qquad Sp(2)\action{\bullet}{\bullet}(Sp(2)\action{\star}{\bullet}(Sp(2)\action{\star}{\bullet}(Sp(2)\action{\star}{\star}(Sp(2)\action{\bullet}{\bullet}Sp(2)))))\to S^7.\]
		\end{enumerate}
		More generally, every combination of $*_i\in \{\bullet,\star\}$ in
		\begin{equation}\label{eq:ladder}
		Sp(2)\action{*_1}{*_2} Sp(2)\action{*_3}{*_4}~\cdots~\action{*_{2k-1}\!\!}{*_{2k}} Sp(2)\to Sp(2)/*_1
		\end{equation} 
		can be realized as a bundle over $Sp(2)/*_1$, as long as  $*_{2m}\neq *_{2m+1}$. By recursively applying Theorem \ref{conjecture}, one concludes that \eqref{eq:ladder} admits a metric with positive Ricci curvature. Indeed, assume that $P\to M$ is a principal $G$-bundle  with an additional  $H$-action, commuting with $G$. Then, $P$ admits an $H$-invariant connection $1$-form  (which is easily obtained by an average argument. More details can be found in \cite{SperancaCavenaghiPublished}). Moreover, the canonical variation of an $H$-invariant bundle metric  is again $H$-invariant, if the horizontal space is given by an $H$-invariant form.

		%	    \begin{equation}
		%	        \xymatrix@C=1.2em@R=1.4em{
		%	        Sp(2) \ar[d] &\\
		%         Sp(2)\times_{S^3}Sp(2) \ar[d] \\
		%            \Sigma^7 \ar[r] &   Sp(2)\times_{S^3}\Sigma^7 \ar[d]\\
		%           &  S^7 \ar[r] &Sp(2)\times_{S^3}S^7 \ar[d]\\
		%          & &  \Sigma^7 \ar[r] & Sp(2)\times_{S^3}\Sigma^7 \ar[d]\\
		%          & & & \Sigma^7}
		%	    \end{equation}
	\end{enumerate}
	%\item 
	\subsubsection{} The construction in the last item can be carried out for any bundle $P$ satisfying \eqref{eq:CD}, simply by replacing $Sp(2)$ by $P$ in \eqref{eq:ladder}:
	\begin{equation}\label{eq:ladderP}
	P\action{*_1}{*_2} P\action{*_3}{*_4}~\cdots~\action{*_{2k-1}\!\!}{*_{2k}} P\to P/*_1,
	\end{equation} 
	as far as  $*_{2m}\neq *_{2m+1}$, $m=1,...,k-1$. Or, even more generally, any family $P_0,...,P_k$, where $P_i$ is equipped with a $G_i\times H_i$-action can be realized as a clustered bundle, as in \eqref{cd:stairs}, as long as the $G_i$- and $H_i$-actions are free and  $H_i=G_{i+1}$ (at least, up to isomorphism):
	\begin{equation}\label{eq:ladderPs}
	P_0\action{H_0}{G_1} P_1\action{H_1}{G_2}~\cdots~\action{H_{k-1}}{G_{k}} P_k\to P_0/H_0.
	\end{equation} 
	As an example, $P_0,...,P_k$ can be taken as any combination of the bundles in the $\star$-diagrams related to items $(i),(ii)$ and $(iii)$ is Theorem \ref{thm:llohann}.
	
	The constructions \eqref{eq:ladderP},\eqref{eq:ladderPs} circumvent two issues of different nature.
	On the one hand, to stack one bundle over the typical fiber of another bundle, one tends to lift  transition functions. That is, given  two bundles, $F\hookrightarrow P\to M$ and $E\to F$, one way to produce a bundle over $M$ whose fibers are $E$ is to lift the transition functions (which usually lie from $\rm{Diff}(F)$ to $\rm{Diff}(E)$. On the other hand, if one ignores the interpretation of \eqref{eq:ladderP} as an iterated bundle, the question on whether the quotient $P_0\action{H_0}{G_1} P_1\action{H_1}{G_2}~\cdots~\action{H_{k-1}}{G_{k}} P_k$ admits  a metric of positive Ricci curvature or not would usually not have a clear answer. As in \cite{pro2014riemannian}, the quotient of a Ricci-positively curved manifold does not need to have positive Ricci curvature.
	Both issue are avoided by realizing \eqref{eq:ladderPs} as a bundle.

	\subsection{Wilhelm bundles} Another closely related construction is the $Sp(2,m)\to S^4$ bundle in \cite{wilhelm-lots}. Consider $S^7$ as in \eqref{eq:Sp2} and $S^4$ as the unitary sphere in $\bb R\times \bb H$. Consider the maps $\rm{pr_1},\rm{pr_2}:Sp(2)\to S^7$, $h,\tilde h:S^7\to S^4$, defined by
	\begin{gather*}
	pr_1\begin{pmatrix}
	a & c \\ b & d
	\end{pmatrix}=\begin{pmatrix}
	a \\ b
	\end{pmatrix},\qquad 
	pr_2\begin{pmatrix}
	a & c \\ b & d
	\end{pmatrix}=\begin{pmatrix}
	c \\ d
	\end{pmatrix} \\
	\qquad	h\begin{pmatrix}
	a\\ b
	\end{pmatrix}=\begin{pmatrix}
	|a|^2-|b|^2\\ 2a \bar b
	\end{pmatrix}, \qquad	
	\tilde h\begin{pmatrix}
	a\\ b
	\end{pmatrix}=\begin{pmatrix}
	|a|^2-|b|^2\\ 2 \bar a b
	\end{pmatrix}.
	\end{gather*}
	%\end{multicols}	
	%	and $k(a,b)=(\bar a, \bar b)$. 
	$Sp(2,m)$ is defined as the submanifold:
	\begin{multline*}
	Sp(2,m)=\{(u_1,...,u_m)\in (S^7)^m~|~h(u_1)=\alpha h(u_2),~\tilde h(u_2)=\alpha h(u_3),\\~\tilde h(u_3)=\alpha h(u_4),\ldots,~\tilde h(u_{m-1})=\alpha h(u_m) \}, 
	\end{multline*}
	where $\alpha$ stands for the antipodal map in $S^4$. It is a principal $(S^3)^m$-bundle over $S^4$, with action
	\begin{multline}\label{eq:actionWilhelm}
	(q_1,...,q_{m})\cdot(u_1,u_2,u_3,...,u_m)\\=(u_1\bar q_1,u_2\bar q_2,q_2u_3\bar q_3,...,q_{m-2}u_{m-1}\bar q_{m-1},q_{m-1}u_m\bar q_m  ).
	\end{multline}
	(we referee to \cite{wilhelm-lots} for two other families of actions. We have chosen this particular action in order to imitate the $Spin(4)=S^3\times S^3$-action in $Sp(2)$ through the first coordinates $u_1,u_2$.)
	Recall that the pull-back of $\pi:P\to M$ by $f:N\to M$ can be defined as the bundle $f^*\pi:f^*P\to N$ where 
	\[f^*P=\{(x,p)\in N\times P~|~\pi(p)=f(x) \}, \]
	and $f^*\pi(x,p)=x$. Hence, $Sp(2,m)$ coincides with the bundle obtained by pulling 
	\begin{align*}
	\pi_{m-1}:Sp(2,m-1)&\to S^4 \\ (u_1,...,u_{m-1})&\mapsto \tilde u_{m-1}
	\end{align*}
	back by $\alpha h:S^7\to S^4$. 
	Using this description, since  $\alpha h$ is a principal $S^3$-bundle, we conclude that forgetting $u_m$ produces a principal $S^3$-bundle $Sp(2,m)\to Sp(2,m-1)$.  Thus, a recursive application of Theorem \ref{conjecture} to the sequence 
	\[Sp(2,m)\to Sp(2,m-1)\to \cdots \to Sp(2,3)\to Sp(2)\] 
	guarantees that $Sp(2,m)$ admits a metric with positive Ricci curvature. Nevertheless, 
	
	Next, we take the opportunity to observe that the bundle $Sp(2,m)\to Sp(2,m-1)$ is actually trivial for $m>2$. Moreover, we observe that $Sp(2,m)$ is actually diffeomorphic to $(S^3)^{m-2}\times Sp(2)$. 
	
	To this aim, let us see $Sp(2,m)$ from a different point of view. First of all, observe that $Sp(2)$ can be defined as the subset (compare \cite{rigas-triality}):
	\[Sp(2)=\{(u_1,u_2)\in S^7\times S^7~|~h(u_1)=\alpha h(u_2) \}. \]
	Therefore, we arrive at the following equivalent description of $Sp(2,m)$:
	\begin{multline*}
	Sp(2,m)=\{(u_1,...,u_m)\in (S^7)^m~|~(u_1,u_2),(ku_2,u_3),...,(ku_{m-1},u_m)\in Sp(2) \}, 
	\end{multline*}
	where $k(a,c)^T=(\bar a,\bar c)$. In particular, $Sp(2,m)$ is diffeomorphic to 
	\begin{multline}\label{eq:Sp2'}
	Sp(2,m)'=\{(Q_1,...,Q_m)\in Sp(2)^m~|~k\rm{pr_2}(Q_i)=\rm{pr}_1(Q_{i+1}),~i=1,...,m\}.
	\end{multline}
	Interestingly enough,  $k:S^7\to S^7$ is a linear isometry with positive determinant,  hence homotopic to the identity. 
	%We readily conclude that there is a diffeomorphism between
	%the following pull-back bundles:
	%\begin{gather}
	%\xymatrix{Sp(2)\ar[r]^{\rm{pr_1}}\ar[d]^{\rm{pr}_2} & S^7\ar[d]^{ah} \\ S^7\ar[r]^{h} & S^4  }\qquad\qquad
	%\xymatrix{\overline{Sp(2)}\ar[r]^{\rm{pr_1}}\ar[d]^{\rm{pr}_2} & S^7\ar[d]^{ah} \\ S^7\ar[r]^{hk} & S^4  }
	%\end{gather}
	%To conclude that $Sp(2,m)\cong(S^3)^{m-2}\times Sp(2)$, 
	%We readily conclude that $Sp(2,3)$
	%we first consider $Sp(2,3)$. For simplicity, we consider $Sp(2,3)'=\{((u_1,u_2),(u_3,\bar u_2))\in Sp(2)^2\}$. First, note that . 
	%Denote $T(u_1,u_2)=(u_2,u_1)$ 
	
	To begin with the $Sp(2,3)\to Sp(2)$ case, observe that  $Sp(2,3)'$ is the pull-back bundle in the first diagram below. Moreover, it is isomorphic to the one in the second diagram, since $k$ is homotopic to the identity:
	\begin{gather*}
	\xymatrix{Sp(2,3)\ar[r]\ar[d] & Sp(2)\ar[d]^{\rm{pr}_1} \\ Sp(2)\ar[r]^{k\rm{pr}_2} & S^7~ ,  }\qquad\qquad \xymatrix{\overline{Sp(2,3)}\ar[r]\ar[d] & Sp(2)\ar[d]^{\rm{pr}_1} \\ Sp(2)\ar[r]^{\rm{pr}_2} & S^7~.  }
	\end{gather*}
	On the other hand, any of the above definitions of $Sp(2)$ guarantees that $(u_1,u_2)\in Sp(2)$ if and only if $(u_2,u_1)\in Sp(2)$. Therefore, 
	\[((u_1,u_2),(u_2,u_1))\in \overline{Sp(2,3)} \] is a section for the principal $S^3$-bundle $\bar\pi:\overline{Sp(2,3)}\to Sp(2)$, guaranteeing its triviality. Since $Sp(2,3)$ and $\overline{Sp(2,3)}$ are isomorphic, one concludes that $Sp(2,3)$ is trivial. 
	
	By keeping track of the $Sp(2,3)$-pull-back diagram, we conclude that there is a map $F:Sp(2)\to Sp(2)$ such that $(Q,F(Q))\in Sp(2,3)$. More generally, 
	\[(Q,F(Q),F^2(Q),...,F^{m-1}(Q))\in Sp(2,m)' \]
	is a section for \eqref{eq:actionWilhelm}. Finally,
	we can define the diffeomorphism $\Phi:(S^3)^{m-2}\times Sp(2)\to Sp(2,m)'$ by
	\[\Phi(q_3,...,q_{m},Q)=(1,1,q_3,...,q_m)\cdot(Q, F(Q),F^2(Q),F^3(Q),...,F^{m-1}(Q) ). \]
	
	Finally, we recall that $Sp(2,m)$, equipped with the inherited metric from $(S^7)^m$, admits an isometric $(S^3)^{m-1}$-action whose quotient is the $(m,-n)$-Milnor bundle, for each $0\leq n\leq m$. Therefore, the diffeomorphism $Sp(2,m)\cong(S^3)^{m-2}\times Sp(2)$ could be taken with some surprise, since there are no obvious actions on $(S^3)^{m-2}\times Sp(2)$ with such quotients. By applying Theorem \ref{conjecture} once again, we conclude that not only has $(S^3)^{m-2}\times Sp(2)$ a Ricci-positive metric invariant by such unusual action, but also every bundle associated to it whose fibers are compatible with Theorem \ref{conjecture}. A simple application of the construction produces a  $(\Sigma^7)^{m-1}$-bundle over $\Sigma^7$, where $\Sigma^7$ is any exotic 7-sphere. 
	
	%    \item
	\subsection{Shrinking Ricci Solitons}
	In \cite{dancer}, examples of compact Shrinking Ricci Solitons with positive Ricci curvature were constructed:
	
	\begin{thm}[Dancer--Wang]\label{thm:dancer1}
		Let $(V_i,J_i,h_i), 1\leq i\leq r,~r\geq 3,$ be Fano K\"ahler--Einstein manifolds with complex dimension $n_i$ and first Chern class $p_i\alpha_i,$ where $p_i > 0,$ and $\alpha_i$ are indivisible classes on in $H^2(V_i;\mathbb{Z}).$ Let $V_1$ and $V_r$ be complex projective spaces with normalized Fubini--Study metric. Let $P_q$ denote the principal $S^1$--bundle over $V_1\times V_2\times \ldots V_r$ with Euler class $-\pi_1^{\ast}(a_1) + \sum_{i=2}^{r-1}q_i\pi_i^{\ast}(a_i) + \pi_r^{\ast}(a_r).$
		
		Suppose in addition that $-(n_1+1)q_i < p_i$ and $(n_r+1)q_i < p_i,~\forall 2\leq i \leq r-1.$ Then, there is a compact shrinking gradient K\"ahler Ricci soltion structure on the total space $\overline M$ obtained from $P_q\times_{S^1}\mathbb{CP}^1$ by blowing $P_q$ down to $V_2\times \ldots\times V_r$ at one end and $V_1\times \ldots \times V_{r-1}$ at the other end, iff for some $\kappa_1 \in \mathbb{R},$ the integral
		\[I := \int_{n_r-1}^{n_r+1}e^{-2\kappa_1(x+n_1+1)}\prod_{i=1}^r\left(x-\dfrac{p_i}{q_i}\right)^{n_i}dx\]
		vanishes. If $\kappa_1 \neq 0,$ the soliton is non--Einstein.
	\end{thm}
	
	The behavior of the integral $I$ on Theorem \ref{thm:dancer1} makes them conclude that
	\begin{thm}[Dancer--Wang]\label{thm:dancer2}
		Every compact manifold $\overline M$ described in Theorem \ref{thm:dancer1} admits a structure of gradient shrinking K\"ahler--Ricci soliton $(g,f)$. Furthermore, $(\overline M,g,f)$ is a Fano manifold.
	\end{thm}
	
	Since Fano manifolds have positive Ricci curvature, we can use Theorem \ref{thm:dancer2} to build bundles with positive Ricci curvature based on a construction with warped products, following Corollary \ref{thm:trivial}. In fact, one has:
	\begin{thm}
		Let $(M,g,f)$ be a non--Einstein compact gradient shrinking K\"ahler Ricci soltion given by Theorem \ref{thm:dancer1}. Let $F$ be a compact manifold with positive Ricci curvature. Then the Riemannian manifold $M\times_{e^f}F$ has positive Ricci curvature after performing a canonical variation.
	\end{thm}
	
	One can proceed to give more explicit examples. In fact, in \cite[pg 21]{dancer}, the following explicit examples of SRS with positive Ricci curvature are constructed:
	\begin{enumerate}
		\item $(M,g,f)$ as a non--trivial shrinking K\"ahler--Ricci soliton on the total space of a $\mathbb{CP}^1$--bundle over $\mathbb{CP}^2\times \mathbb{CP}^2,$
		\item $(M,g,f)$ as a non--trivial shrinking K\"ahler--Ricci soliton on the total space of a $\mathbb{CP}^1$--bundle over a product of two copies of $\mathcal{F}(3,3),$ where $\mathcal{F}(3,3)$ is the irrational Clemens-Griffiths three-fold.
	\end{enumerate}
	
	Let $F$ be any example of a manifold with positive Ricci curvature constructed with the methods outlined above. In particular, one can simply consider $F$ as any manifold satisfying the hypotheses in Theorem \ref{thm:llohann}. Then, the following Riemannian manifolds have positive Ricci curvature after performing a canonical variation:
	\begin{enumerate}
		\item $\widehat M_1 := M\times_{e^{2f}} F$, where $(M,g,f)$ is a non--trivial shrinking K\"ahler--Ricci soliton on the total space of a $\mathbb{CP}^1$--bundle over $\mathbb{CP}^2\times \mathbb{CP}^2,$
		\item $\widehat M_2 := M\times_{e^{2f}} F,$ where $(M,g,f)$ is a non--trivial shrinking K\"ahler--Ricci soliton on the total space of a $\mathbb{CP}^1$--bundle over a product of two copies of $\mathcal{F}(3,3),$ where $\mathcal{F}(3,3)$ is the irrational Clemens-Griffiths three-fold,
		\item Furthermore, one can proceed indefinitely by taking $\widehat M_{i+1} := M\times_{e^{2f}}\widehat M_i,~i\in \{1,2,\ldots\},$ and construct new examples of manifolds with positive Ricci curvature. Here, $(M,g,f)$ can be taken as any of the examples in the previous items. 
	\end{enumerate}
	%            \end{enumerate}
	
	% ------------------------------------------------------------------------

	\subsection*{Acknowledgments}
	The authors are grateful to the anonymous referees for the enormous patience on reading (several) earlier versions of this manuscript and giving them precise directions on how to improve it. The first author would like to thank Augusto Pereira, Gustavo Costa and Renato Júnior for the criticism and help with the implementation of the referees suggestions, as well to the complete English review. He also thank Prof. Marcos Alexandrino, Prof. Francisco Caramello and Daniel Fadel for the encouragement and for their useful comments on earlier versions of this paper.
	
	\bibliographystyle{alpha}
	\bibliography{bjourdoc}
	
\end{document}